\documentclass[11pt]{amsart}
\usepackage{amsmath,amssymb,amsfonts}
\usepackage{tikz-cd}
\usepackage{amsbsy}
\usepackage{amscd}
\usetikzlibrary{matrix,arrows,decorations.pathmorphing}
\usepackage{graphicx}
\usepackage{float}
\usepackage{enumitem}
\usepackage{setspace}
\usepackage[margin=1.0in]{geometry}
\usepackage[all]{xy}

\usepackage{yfonts}

\usepackage{xcolor}
\definecolor{blue1}{RGB}{32,78,170}
\definecolor{blue2}{RGB}{93,92,160}
\definecolor{blue3}{RGB}{40,51,202}
\definecolor{blue4}{RGB}{0,0,0}
\definecolor{purple1}{RGB}{128,0,128}

\usepackage[all]{xy}

\definecolor{El}{rgb}{.4,.9,1}
\usepackage[pdftex,
            linkcolor=blue1,
            citecolor=blue3,
            colorlinks=true]{hyperref}

\usepackage{titlesec}

\titleformat{\section}
  {\normalfont\fontsize{12}{15}\bfseries}{\thesection}{1em}{}

\usepackage[english]{babel}
\usepackage{blindtext}
\usepackage[scaled=.90]{helvet}
\usepackage{xcolor}
\usepackage{titlesec}
\titleformat{\chapter}[display]
  {\normalfont\sffamily\huge\bfseries\color{blue4}}
  {\chaptertitlename\ \thechapter}{20pt}{\Huge}
\titleformat{\section}
  {\normalfont\sffamily\Large\color{blue4}}
{\thesection}{1em}{}

 \titleformat{\subsection}
  {\normalfont\sffamily\large\color{blue4}}
  {\thesubsection}{1em}{}

\newcommand{\ma}[1]{\emph{#1}}

            \newcounter{pulse}[section]

\numberwithin{pulse}{section}
\numberwithin{equation}{section}

\newtheorem{theorem}[pulse]{\bf \textsf{Theorem}}
\newtheorem{proposition}[pulse]{\bf \textsf{Proposition}}

\newtheorem{lem}[pulse]{\bf \textsf{Lemma}}
\newtheorem{cor}[pulse]{\bf \textsf{Corollary}}

\newtheorem{dummy-eg}[pulse]{\bf \textsf{Example}}
\newtheorem{dummy-rem}[pulse]{\bf \textsf{Remark}}
\newenvironment{eg}{\begin{dummy-eg}\upshape}{\end{dummy-eg}\ignorespacesafterend}

\newtheorem{dummy-def}[pulse]{\bf \textsf{Definition}}

\usepackage{amsmath,amssymb}

\newenvironment{dfn}{\begin{dummy-def}\upshape}{\end{dummy-def}\ignorespacesafterend}

\newcommand{\supp}{\operatorname{supp}}

\newcommand{\norm}[1]{\Vert #1 \Vert}

\newcommand{\VN}{\operatorname{VN}}
\newcommand{\ignore}[1]{}

\title{Remarks on weak amenability of hypergroups}
\author{Mahmood Alaghmandan}
\email{mahmood.alaghmandan@carleton.ca}
\address{School of Mathematics and Statistics, Carleton University, Ottawa, ON, Canada H1S 5B6}


\begin{document}

\maketitle

\vskip1em

\begin{center}
\large{{Dedicated to Rupert Lasser on the occasion of his 70th birthday.}}
\end{center}

\begin{abstract} 
We study the existence of multiplier (completely) bounded   approximate identities for the Fourier algebras of some classes of  hypergroups. In particular we show that, a large class of commutative hypergroups are weakly amenable with the Cowling-Haagerup constant $1$.  As a corollary, we  answer an open question of Eymard on Jacobi hypergroups. We also characterize the existence of  bounded approximate identities for the hypergroup Fourier algebras of ultraspherical hypergroups.
\end{abstract}

\section{Introduction}

In \cite{me}, Meaney studied the spectral synthesis properties of the Fourier algebras  of \emph{Jacobi hypergroups}. These are  hypergroup structures $H_{\alpha, \beta}$ on $\Bbb{R}^+$ defined by Jacobi functions where    $\alpha \geq \beta \geq -1/2$ but not $ \alpha = \beta = -1/2$.

Since for   some values of  $\alpha$ and $\beta$, the Jacobi hypergroups are isomorphic to some double coset structures on locally compact groups, they have been  of interest to harmonic analysts. For example for pairs $(\alpha = 2n -1, \beta = 1)$ and  $(\alpha = 7, \beta =3)$, $H_{\alpha, \beta}$ is isomorphic to some double coset hypergroups on $\operatorname{Sp}(n,1)$  and $F_{4(-20)}$ respectively. 
Pointing out the role of this double coset structure in studying the Cowling-Haagerup constants of $\operatorname{Sp}(n,1)$ and $F_{4(-20)}$ (\cite{mi}),  Eymard (\cite{eym-sur})  asks  for the analogues of the completely bounded multipliers for the Fourier algebra  and the Cowling-Haagerup constant for  $H_{\alpha, \beta}$. He was in particular interested to know, if for $H_{\alpha, \beta}$, such a constant would change by parameters $\alpha, \beta$.

The spaces of bounded multipliers and then later  completely bounded multipliers of hypergroup Fourier algebras  have been studied  in \cite{ac2,mu1,mu2}. The notions of weak amenability and subsequently the Cowling-Haagerup constant of a hypergroup were introduced  in \cite{ac2} where Crann and the author showed that  for every discrete commutative hypergroup, this constant is $1$. This result is not enough to answer Eymard's question, as commutative hypergroups $H_{\alpha,\beta}$ are not discrete. 
In this short paper, we answer Eymard's question  (Corollary~\ref{c:Jacobi-WA}) by  showing that the Cowling-Haagerup constant of $H_{\alpha, \beta}$ is   $1$ for all values of  $\alpha, \beta$.  

The paper is organized as follows. 
Section~\ref{s:notation} briefly introduces our notation. 
In Section~\ref{s:P-l}, we study the notion of weak amenability for a larger class of hypergroups so-called  P$_\lambda$-hypergroups. These hypergroups behave well with respect to multiplication of their reduced positive definite functions. 
In this Section, we use this feature of these hypergroups as well as a careful application of a study by Voit (\cite{voit}) on commutative hypergroups to prove that every commutative P$_\lambda$-hypergroup is weakly amenable with constant $1$. 
In Section~\ref{s:ultra}, we focus our study on  ultraspherical hypergroups. This class of hypergroups covers double cosets hypergroups as a  particular subclass and therefore it has been of interest in studying locally compact groups. We conclude this section with a complete Leptin theorem for this class of hypergroups which characterizes the existence of bounded approximate identities  for the Fourier algebras of ultraspherical hypergroups.

\section{Notation}\label{s:notation}

Here $H$ is a hypergroup with $(M(H), \cdot)$ the algebra of all bounded measures on $H$. 
 All hypergroups considered in this article are assumed to possess a left invariant Haar measure $\lambda$, and we use $\int_H dx$ to denote the integration with respect to $\lambda$ if there is no risk of confusion.
 
 We let $(L^p(H), \norm{\cdot}_p)$, $1 \leq p \leq  \infty$, denote the usual $L^p$-space with respect to a fixed left Haar measure $\lambda$. For each $f \in L^1(H)$ and $g \in L^p(H)$ for $1 \leq p \leq \infty$, it follows that
 \[
 f \cdot_\lambda g (x):=\int_H  f( y)g(\tilde{y} \cdot x )dy
 \]
 belongs to $L^p(H)$ and $\norm{f \cdot_\lambda  g}_p \leq  \norm{f}_1 \norm{g}_p$.
 In particular,  $L^1(H)$ is a Banach $\ast$-algebra with this action and  the involution  $f^*(x)=\Delta(\tilde{x})f(\tilde{x})$ where $\Delta$ is the modular function with respect to the left Haar measure $\lambda$.

 A hypergroup $H$ is called \ma{commutative} if $\nu_1 \cdot \nu_2= \nu_2 \cdot \nu_1$ for every pair $\nu_1,\nu_2 \in M(H)$.
We define $\widehat{H}:=\{ \chi \in C_b(H): \chi(x \cdot y) = \chi(x) \chi(y)\ \text{and}\ \chi(\tilde{x})= \overline{\chi(x)}\}$,
equipped with the topology of uniform convergence on compact subsets of  $H$. One can  define Fourier and Fourier-Stieltjes transforms on $L^1(H)$ and $M(H)$ respectively, similar to the group case. In particular, there is a positive measure $\varpi$ on $\widehat{H}$, called the \ma{Plancherel measure}, such that the Fourier transform can be extended to a isometric isomorphism from  $L^2(H)$ onto $L^2(\supp(\varpi), \varpi)$.  Unlike the group case, $\widehat{H}$ does not necessarily form a hypergroup and $\supp(\varpi)$ may be a proper subset of $\widehat{H}$.

Abusing notation, we  let $\lambda$ also denote the left regular representation of $H$ on $L^2(H)$ given by
$\lambda(x)\xi(y)=\xi (\tilde{x}\cdot y)$ for all $x,y\in H$ and for all $\xi \in L^2(H)$. This can be extended to $L^1(H)$ where $\lambda(f)\xi:=f \cdot_\lambda \xi$ for $f\in L^1(H)$ and $\xi \in L^2(H)$.
Let $C^*_\lambda(H)$ denote the  completion of $\lambda(L^1(H))$ in $\mathcal{B}(L^2(H))$ which is called the \ma{reduced $C^*$-algebra} of $H$. The von Neumann algebra generated by $\{\lambda(x): x \in H\}$ is called the \ma{von Neumannn algebra} of $H$, and is denoted by $\VN(H)$.

Let $P_\lambda(H)$ denote the set of all bounded, continuous positive definite functionals on $C^*_\lambda(H)$.  By identifying $P_\lambda(H)$ as functions on $H$, one can span $P_\lambda(H)$ to generate $B_\lambda(H)$, the reduced Fourier-Stieltjes space of $H$. On can show that $B_\lambda(H)$ is the dual space of $C^*_\lambda(H)$. The Fourier space, $A(H)$, is the closed subspace of $B_\lambda(H)$ which is spanned by compactly supported elements in $P_\lambda(H)$ and is  isomorphic to the predual of $VN(H)$.  Note that in general, $A(H)$ is not a Banach algebra. If this is the case, $H$ is called a {\em regular Fourier hypergroup}. 

As the predual of $VN(H)$, $A(H)$ has a canonical operator space structure.
A hypergroup $H$ is called   a \emph{completely Fourier hypergroup} if $A(H)$, furnished with its canonical operator space structure,  is a completely contractive Banach algebra  (look at \cite[Section~3]{ac2}). 
 Based on this operator space structure, the space of completely bounded multipliers of  $A(H)$, denoted by $M_{cb}A(H)$,  for a completely Fourier hypergroup $H$ is defined. 
 It can be shown that  a commutative regular Fourier  hypergroup $H$ is always a completely Fourier hypergroup and  the multiplier algebra of $A(H)$, denoted by $MA(H)$, is equal to  $M_{cb}A(H)$. 
 
 A completely Fourier hypergroup $H$ is called \emph{weakly amenable} with the (Cowling-Haagerup) constant $C$ is there is an approximate identity $(u_\alpha)$ for $A(H)$ so that $\sup_\alpha \| u_\alpha\|_{M_{cb}A(H)} \leq C$.

 For more on  hypergroup Fourier and  reduced Fourier-Stieltjes spaces, in particular on  commutative and ultraspherical hypergroups, we refer the reader to \cite{ma5, ac2, mu1, mu2}.

 \section{P$_\lambda$-hypergroups and their Leptin theorem}\label{s:P-l}

\begin{dfn}\label{d:P-hypergroups}
We call a hypergroup $H$   a \emph{P$_{\lambda}$-hypergroup} if $P_\lambda(H)$   is closed under pointwise multiplication.
\end{dfn}

Based on an argument in \cite[Corollary~4.13]{mu1}, one can observe that for a  commutative P$_\lambda$-hypergroup $H$, $A(H)$ and $B_\lambda(H)$  are forming Banach algebras with  their own norms. Formerly, this was proved in \cite[Theorem~2.2]{am}  where this feature was studied for hypergroups $H$ where $P(H)$, the set of all bounded, continuous positive definite functions, is closed under pointwise multiplication.  The later proof could be easily modified for P$_\lambda$-hypergroup. Therefore, every P$_\lambda$-hypergroup is a regular Fourier hypergroup. 

 The following proposition helps us to find more examples of commutative P$_\lambda$-hypergroups.  

\begin{proposition}\label{p:P-lambda-characters}
Let $H$ be a commutative hypergroup so that for every pair $\chi, \psi$ in $\supp(\varpi)$, $\chi \psi$ belongs to $P_\lambda(H)$. Then $H$ is a $P_\lambda$-hypergroup.
\end{proposition}

\begin{proof}
Let $S:=\supp(\varpi)$.  It is known that $C^*_\lambda(H)$ is isometrically isomorphic to $C_0(S)$ and $B_\lambda(H)$ is isometrically isomorphic to $M(S)$ through the Fourier Stieltjes transform (\cite{mu1}).  Recall that $B_\lambda(H)$ is a dual Banach space i.e. for each $u \in B_\lambda(H)$ and $f \in C^*_\lambda(H)$, $u \cdot f \in C^*_\lambda(H)$ where  $\langle u \cdot f, w\rangle := \langle f, uw\rangle$ ($w \in B_\lambda(H)$). 
 
 Let $f \in C^*_\lambda(H)$ and $u$ belong to $B_\lambda(H)$. Let $\mu \in M(S)$ be the Fourier Stieltjes transform of $u$.  Hence, for each $\psi \in S$ we have
 \[
 \langle u \cdot f, \psi\rangle = \langle \psi \cdot f, u\rangle = \int_S \langle \psi \cdot f, \bar{\phi}\rangle d\mu(\phi) = 
  \int_S \langle \bar{\phi} \cdot f,  \psi\rangle d\mu(\phi) =  \left\langle \int_S \bar{\phi} \cdot f d\mu(\phi), \psi\right\rangle.
 \]
Now let $u, v \in P_\lambda(G)$ and $\mu, \nu \in M(S)^+$  where $\mu$ and $\nu$ are the  Fourier Stieltjs  transforms of $u$ and $v$ respectively.  Then for each positive element $f \in C^*_\lambda(H)^+$we get
 \begin{eqnarray*}
 \langle f, uv\rangle = \langle u \cdot f, v \rangle =  \langle \int_S  \bar{\phi} \cdot f d\mu(\phi), v\rangle
 = \int_S \int_S \langle \bar{\phi} \cdot f, \bar{\psi}\rangle  d\nu(\psi) d\mu(\phi)
 = \int_S \int_S \langle f, \bar{\psi}\;  \bar{\phi}\rangle   d\nu(\psi) d\mu(\phi).
 \end{eqnarray*}
 But by the assumption of the proposition, $ \langle f, \bar{\psi}\;  \bar{\phi}\rangle\geq 0$ for every pair $\phi, \psi \in S$.  Now the positivity of the measures $\mu, \nu$ implies that $\langle f, uv\rangle \geq 0$ which finishes the proof.
\end{proof}

\begin{eg}\label{eg:Jacobi-hypergroups}
Let $H_{\alpha, \beta}$ denote the Jacobi hypergroup corresponded to parameters $\alpha, \beta$. 
It was proved   in \cite[Section 4]{jac} that, for every pair $\phi, \psi$ in the support of the Plancherel measure, the function $\phi\psi$ belongs to $P_\lambda(H)$. 
Hence by Proposition~\ref{p:P-lambda-characters},  $H_{\alpha, \beta}$ is a P$_\lambda$-hypergroup.  
\end{eg} 

In the following, we apply a configuration for commutative hypergroups which is due to Voit. In \cite{voit}, he proves that for every commutative hypergroup $H$, there is a unique positive character $\chi_0$ in the support of the Plancherel measure satisfying $\sup_{\chi \in \supp(\varpi)}  |\chi(x)| \leq \chi_0(x)$
for every $ x \in H$.  He then defines a new hypergroup convolution on $H$, denoted here by $\circ$, via
\[
  x  \circ y := \frac{ \chi_0}{\chi_0( x \cdot y)} x \cdot y
\]
for every pair $x,y \in H$. We use $H_0$ to denote the later hypergroup structure on $H$ whose Haar measure is $\lambda'=\chi_0^2 \lambda$. In this case, it follows that $\widehat{H}_0$ is homeomorphic to the set $\{ \chi \in \widehat{H}: |\chi(x)| \leq \chi_0(x)\ \ \forall x\in H\}$ through the mapping $\chi \mapsto \chi/\chi_0$.  

 In \cite{ac2}, it was shown that $u \mapsto u/\chi_0$ forms an isometry from $A(H)$ onto $A(H_0)$ as Banach spaces. In fact, it was observed that $A(H_0)$ is isometric to the closure of $A(H)$ in its multiplier norm.  In the following we study $P_\lambda(H_0)$ for a commutative P$_\lambda$-hypergroup.

\begin{proposition}\label{p:P-lambda(H0)}
Let $H$ be a commutative  hypergroup with the positive character $\chi_0$ in the support of the Plancherel measure. Then $v\chi_0 \in P_\lambda(H)$ if and only if $v \in P_\lambda(H_0)$. Moreover, if $H$ is  a  P$_\lambda$-hypergroup, then
for each pair $v, u \in P_\lambda(H_0)$,  both of $\chi_0 uv$ and $u\chi_0$ belong to $P_\lambda(H_0)$ and  $P_\lambda(H) \subseteq P_\lambda(H_0)$.
\end{proposition}

\begin{proof}
First note that by \cite[Lemma~3.8]{ac2}, $v\in B_\lambda(H_0)$ if and only if $\chi_0v \in B_\lambda(H)$. To show that  $P_\lambda(H_0)$, it is enough to show that for each $f \in C_c(H)$, $\langle u, f \circ_{\lambda'} f^*\rangle \geq 0$. The very same holds  for $\chi_0 v$, that is $\chi_0 v \in P_\lambda(H)$ if $\langle u, f \cdot_{\lambda} f^*\rangle \geq 0$. But by the computations in the proof of \cite[Lemma~3.8]{ac2}, we know that $f \circ_{\lambda'} f^* = \left((\chi_0 f) \cdot_\lambda (\chi_0 f)^* \right)/{\chi_0}$ for every  $f \in C_c(H)$.
Therefore, we get
\begin{eqnarray*}
\langle v, f \circ_{\lambda'} f^*\rangle  &=&  \int_H v(x) f \circ_{\lambda'} f^*(x) d\lambda'(x)\\
&=& \int_H (\chi_0v)(x)  (\chi_0 f) \cdot_{\lambda} (\chi_0 f)^*(x)   d\lambda(x)= \langle \chi_0 v, (\chi_0 f) \cdot_{\lambda} (\chi_0 f)^*\rangle.
\end{eqnarray*}
Since $f \mapsto \chi_0 f$ is a Banach space isomorphism from $C_c(H)$ onto itself, we have that $v \in P_\lambda(H_0)$ if and only $v\chi_0 \in P_\lambda(H)$. 

 Let $v, u \in P_\lambda(H_0)$. By what we proved above, we know that $\chi_0 v, \chi_0 u \in P_\lambda(H)$. Since $H$ is a P$_\lambda$-hypergroup, $\chi_0^2 vu \in P_\lambda(H)$. Again by the argument above, we get $\chi_0vu \in P_\lambda(H_0)$. 
 Note that   the constant character $1$ belongs to $P_\lambda(H_0)$. So in particular, for each $u \in P_\lambda(H_0)$, $\chi_01 u = \chi_0 u$ belongs to $P_\lambda(H_0)$.
Eventually, let $w \in P_\lambda(H)$. Then there exists  $v \in P_\lambda(H_0)$ so that $\chi_0 v=w$. But  $\chi_0 v \in P_\lambda(H_0)$.    
\end{proof}

 The following result is an analogue of \cite[Theorem~3.7]{ac2} for P$_\lambda$-hypergroups. The main idea of the proof is inspired by the proof of the Nielson lemma presented in \cite[Proposition~5.1]{haa} with an application of  Voit's theory for commutative hypergroups. 

\begin{theorem}\label{t:P-voit}
Let $H$ be a commutative   P$_\lambda$-hypergroup. Then $H$ is weakly amenable with constant~$1$.
\end{theorem}

\begin{proof}
 Recall that a hypergroup is said to satisfy $(P_2)$ if there exists a Reiter's net in $L^2(H)$, see \cite{sk}.
Since $H_0$ satisfies $(P_2)$, there is a net $(u_\alpha)$ in $P_\lambda(H_0) \cap C_c(H)$ so that $u_\alpha \rightarrow 1$ uniformly on compact sets of $H$ and $\norm{u_\alpha}_{A(H_0)}=u_\alpha(e)=1$ for each $\alpha$, \cite[Lemma~4.4]{sk}. 
Let $w \in P_\lambda(H) \cap C_c(H)$. By Proposition~\ref{p:P-lambda(H0)}, $w \in P_\lambda(H_0)\cap C_c(H)$.
Define $w_\alpha =   u_\alpha w$  for each $\alpha$. Since $w_\alpha= \chi_0 vu_\alpha$ for some $v \in P_\lambda(H_0)$,  $(w_\alpha)$ is a subset of  $P_\lambda(H_0) \cap C_c(H)$ so that for each $\alpha$, $\supp(w_\alpha) \subseteq \supp(w)$.  Since $u_\alpha \rightarrow 1$ uniformly on compact sets, 
\[
\lim_\alpha \norm{ w_\alpha - w}_\infty =0.
\]

Since $w \in P_\lambda(H_0) \cap C_c(H) \subseteq A(H_0)$ is a positive form on $C^*_\lambda(H_0)$, there is some $h \in L^2(H_0)$ so that $w = h \circ_{\lambda'} h^*$.
Indeed, $\lambda'(h)f = \lambda'(w)^{1/2}f$ for each $f \in C_c(H)$ where $\lambda'(h)$ and $\lambda'(w)$ denote the left regular representations of $h$ and $w$ on $L^2(H_0)$. Similarly for each $\alpha$, there is some $h_\alpha \in L^2(H_0)$ so that $w_\alpha = h_\alpha  \circ_{\lambda'} h_\alpha $  and hence  $\lambda'(h_\alpha)f= \lambda'(w_\alpha)^{1/2}f$ for $f\in C_c(H)$.
 
 For each $f \in C_c(G)$, note that 
\begin{eqnarray*}
\norm{(w_\alpha - w) \circ_{\lambda'} f}_2 &=& \left\| \int_H (w_\alpha(y) - w(y)) L_yf d\lambda'(y)\right\|\\
&\leq& \left( \int_H |w_\alpha(y) - w(y)|d\lambda'(y)\right) \norm{f}_2\\
&\leq& \lambda'(\supp(w)) \norm{w_\alpha - w}_\infty \norm{f}_2. 
\end{eqnarray*}
Therefore, $\norm{\lambda'(w_\alpha) - \lambda'(w)}\rightarrow 0$ and hence $c= \sup_\alpha \norm{\lambda'(w_\alpha)}<\infty$.  
 By approximating the function $t \mapsto \sqrt{t}$ with polynomials uniformly on $[0,c]$, we have
 $\norm{ \lambda'(w_\alpha)^{1/2} - \lambda'(w)^{1/2} } \rightarrow 0$. 
Consequently, 
\[
\lim_\alpha \norm{(h_\alpha - h)* f}_2 = \lim_\alpha\norm{   (\lambda'(w_\alpha)^{1/2} - \lambda'(w)^{1/2})f }_2 \rightarrow 0
\]
for every $f \in C_c(H)$. Consequently, 
\[
\lim_\alpha \langle h_\alpha - h, f*g\rangle = \lim_\alpha \langle ( h_\alpha - h)* f^*, g\rangle = 0
\]
for all $f,g \in C_c(H)$. Since $C_c(H) \circ_{\lambda'} C_c(H)$ is dense in $L^2(H_0)$ and since $\sup_\alpha \|h_\alpha\|_2^2 = \sup_\alpha w_\alpha(e) < \infty $, it follows that 
$\lim_\alpha \langle h_\alpha, f\rangle = \langle h, f\rangle$  for all $f \in L^2(H_0)$.
 Since $\lim_\alpha \norm{h_\alpha}^2_2 = \lim_\alpha w_\alpha(e) = w(e) = \norm{h}_2^2$, one concludes that 
 \[
 \lim_\alpha \norm{w_\alpha - w}_2^2 = 2 \norm{h}_2^2 - 2 \lim_\alpha \text{Re} \langle h_\alpha, h\rangle = 0.
 \]
Therefore, we have
\begin{eqnarray*}
\norm{u_\alpha w - w}_{A(H_0)} &=& \norm{w_\alpha - w}_{A(H_0)}\\
 &=& \norm{h_\alpha \circ_{\lambda'} h_\alpha^* - h \circ_{\lambda'} h^*}_{A(H_0)}\\
&=& \norm{ h_\alpha \circ_{\lambda'} (h_\alpha^* - h^*) + (h_\alpha - h) \circ_{\lambda'} h_\alpha^*}_{A(H_0)}\\
&\leq& \norm{h_\alpha - h}_2 (\norm{h_\alpha}_2 + \norm{h}_2) \rightarrow 0.
\end{eqnarray*} 
 
 \vskip1.0em
 
Note that by \cite[Remark~3.12]{ac2}, in fact we have 
\[
\lim_\alpha\norm{u_\alpha w   - w}_{MA(H)} = 0 \quad \quad (w \in P_\lambda(H) \cap C_c(H)).
\]
Also note that $(u_\alpha)$ is a net in $A(H_0) \cap C_c(H)$. But this set is equal to $A(H) \cap C_c(H)$, by \cite[Remark~3.9 (2)]{ac2}. 
Let $\mathcal{A}$ be the linear span of $P_\lambda(H) \cap C_c(H)$, which is a dense subspace of $A(H)$.
Therefore, $(u_\alpha)$ is an approximate identity for $(\mathcal{A}, \norm{\cdot}_{MA(H)})$ whose $\norm{\cdot}_{MA(H)}$-norm is bounded by $1$. 
Since $\mathcal{A}$ is dense in $A(H)$ and $\norm{\cdot}_{MA(H)}\leq \norm{\cdot}_{A(H)}$. We have that  $(u_\alpha)$ is a bounded  approximate identity of $(A(H), \norm{\cdot}_{MA(H)})$.

Now let $u \in A(G) \cap C_c(H)$ be an arbitrary element. By \cite[Lemma~3.4]{ma}, there exists an element $u' \in A(G) \cap C_c(G)$ so that $u'u=u$. For $\epsilon>0$, chose $\alpha_0$ so that for every $\alpha \succeq \alpha_0$, $\norm{uu_\alpha - u}_{MA(H)} < \epsilon / \norm{u'}_{A(H)}$. 
Therefore,
\[
\norm{ uu_\alpha - u}_{A(H)} = \norm{u'(u u _\alpha - u)}_{A(H)} \leq \norm{u'}_{A(H)} \norm{uu_\alpha - u}_{MA(H)} < \epsilon.
\]
But since $C_c(H) \cap A(H)$ is dense in $A(H)$ and $(u_\alpha)_\alpha$ is a bounded net in $\norm{\cdot}_{MA(H)}$, one can easily show that $(u_\alpha)_\alpha$ is indeed an approximate identity for $A(H)$. 
\end{proof}

The following corollary is an application of Theorem~\ref{t:P-voit} to Example~\ref{eg:Jacobi-hypergroups}. 

\begin{cor}\label{c:Jacobi-WA}
Let $H_{\alpha, \beta}$ be the Jacobi hypergroup admitted by parameters $\alpha \geq \beta \geq 1/2$. Then $H_{\alpha, \beta}$ is weakly amenable with constant $1$.
\end{cor}

Note that in the proof of Theorem~\ref{t:P-voit}, we used commutativity of $H$ to have the parallel hypergroup $H_0$ which satisfies $(P_2)$. In fact if $H$ is not a necessarily commutative hypergroup which satisfies $(P_2)$, one can rewrite the proof again by replacing all $H_0$ with $H$. The modified proof would prove   one side of the Leptin theorem for P$_\lambda$-hypergroups which is presented as the following proposition. A similar result was formerly proved for discrete hypergroups in \cite{ma5}.  
 
\begin{proposition}\label{p:Leptin-P-hypergroup}
Let $H$ be a  P$_\lambda$-hypergroup. If $H$ satisfies $(P_2)$, then $A(H)$ has a bounded approximate identity.
\end{proposition}

\section{Ultraspherical hypergroups}\label{s:ultra}

For definitions, notation, and basic results which we use in this section we refer the reader to \cite{mu2}.  Let $H$ be an ultraspherical hypergroup structure defined on  $G$.  Muruganandam, \cite{mu2}, proves the existence of a norm decreasing homorphism
\begin{equation}\label{eq:phi}
\pi: B_\lambda(G) \rightarrow B_\lambda(G)
\end{equation}
whose image is $B_\lambda(G)^{\sharp}$,  the subspace of all  radial elements in $B(G)$. Then he proves that $B_\lambda(G)^{\sharp}$ is isomorphic to the reduced Fourier-Stieltjes algebra of $B_\lambda(H)$. Further, $\pi$ maps positive definite functions into positive definite functions on $H$, through this identification.  The  restriction of $\pi$ to the Fourier algebra $A(G)$ is mapped onto $A(G)^\sharp$, the subalgebra of radial elements in $A(G)$,  which is isomorphic to $A(H)$. 
In this case, the dual of this restriction gives 
\begin{equation}\label{eq:iota}
\iota:= \left( \pi |_{A(G)^\sharp}\right)^*: VN(G)^\sharp \rightarrow VN(H)
\end{equation}
a complete isometry where $VN(G)^\sharp$ is the subalgebra of $VN(G)$ generated by radial elements in $L^1(G)$ (\cite[Theorem~3.9]{mu2}).   This subsequently  implies that   $H$ is a completely Fourier hypergroup, that is, $A(H)$ is a completely contractive Banach algebra with its canonical operator space structure. 

Since $\pi$ defined in (\ref{eq:phi}) is a weak$^*$ continuous map, its predual mapping  is 
\begin{equation}\label{eq:pi*}
\pi_*: C^*_\lambda(G) \rightarrow C^*_\lambda(G)
\end{equation}
whose image is $C^*_\lambda(G)^{\sharp}$,  the norm closure of $\lambda(f)$ for all $f \in L^1(G)$ which are radial. In fact, $\pi_*$ is a contractive projection from $C^*_\lambda(G)$ onto $C^*_\lambda(G)^\sharp$ (or equivalently $C^*_\lambda(H)$). By a known result on projections of $C^*$-algebras  (\cite[Theorem 1.5.10]{oz}), $\pi_*$ is in fact a contractive completely positive map.

The rich structure of ultraspherical hypergroups has made article \cite{ultra} possible where the spectral synthesis sets for different classes of ultraspherical hypergroups were studied and very interesting unexpected results were found.
   
\begin{proposition}\label{p:ultrasphrical-is-P}
Let $H$ be an ultraspherical hypergroup. Then $H$ is a P$_\lambda$-hypergroup. 
\end{proposition}

\begin{proof}
Let $u$ be  chosen arbitrarily in $P_\lambda(H)$. Note that $P_\lambda(H) \subseteq B_\lambda(H)$. Through the identification, $B_\lambda(G)^\sharp = B(H)$,    let $u = \pi(u)$ be  identified with an element  $B_\lambda(G)^\sharp$.
For each $f \in C^*_\lambda(G)^+$,  $\pi_*(f)$ is a positive  element in $C^*_\lambda(G)^\sharp$ since $\pi_*$ is a complete positive map. Subsequently, 
\[
\langle u , f\rangle = \langle u, \pi_*(f)\rangle \geq 0.
\]
This implies that $u$ belongs to $P_\lambda(G)$. Now let $u_1, u_2$ be in $P_\lambda(H)$, again identified as two radial  elements on $G$.  Since $G$ is a locally compact group, it is a P$_\lambda$-hypergroup. Therefore, $u_1 u_2 \in P_\lambda(G)$  is another radial element on $G$; by \cite[Theorem~3.3]{mu2},   there exists  $u \in P_\lambda(H)$ so that $u$ is identified by $u_1 u_2$. 
\end{proof}

The following corollary is an immediate consequence of Proposition~\ref{p:ultrasphrical-is-P} and Theorem~\ref{t:P-voit}.   Recall that a \emph{Gelfand pair} is a pair $G,K$ consisting of a locally compact group $G$ and a compact subgroup $K$ such that the subalgebra of $(K,K)$-double coset invariant subalgebra of $L^1(G)$ is  commutative.  The double coset structure of a Gelfand pair forms a commutative ultraspherical hypergroup.  
 
 \begin{cor}\label{c:WA-ultraspherical}
 Let $H$ be a commutative ultraspherical hypergroup. Then $H$ is  weakly amenable with constant $1$. In particular every Gelfand pair admits a weakly amenable hypergroup with constant $1$.  
 \end{cor}

We finish this paper by  a complete characterization of the existence of  bounded approximate identities for $A(H)$ for an ultraspherical hypergroup $H$. To prove this characterization we need  the following lemma.

\begin{lem}\label{l:G-amenable=>H-P2}
Let $H$ be an ultraspherical hypergroup on a locally compact group $G$. If $G$ is amenable then $H$ satisfies $(P_2)$.
\end{lem}

\begin{proof}
Let $G$ be amenable. Then there is a net $(u_\alpha)_\alpha$ of positive definite compactly supported elements of $P_\lambda(G)$ so that $u_\alpha \rightarrow 1$ uniformly on compact sets of $G$. Let $v_\alpha:=\pi(u_\alpha)$ for each $\alpha$ where $\pi$ is defined in (\ref{eq:phi}). Let us identify $B_\lambda(G)^{\sharp}$ with $B_\lambda(H)$. Then $(v_\alpha)$ forms a net in $P(H) \cap C_c(H)$, by \cite[Theorem~3.3]{mu2}. It is known that $P(H) \cap C_c(H) \subseteq P_\lambda(H) \cap A(H) $, \cite[Corollary~2.12]{mu1}. Therefore, there are nets $(\xi_\alpha)_\alpha$ in $L^2(H)$ with $\norm{\xi_\alpha}_2 =1$ so that $v_\alpha=\xi_\alpha \cdot_\lambda \tilde{\xi}_\alpha$ for each $\alpha$. 

Let $H = \{ \dot{x} : x\in G \}$ denote the ultraspherical hypergroup obtained in \cite[Theorem 2.12]{mu2}. Let $q$ denote the quotient map from $G$ onto $H$.
Let   $K \subseteq H$ be compact  and $\epsilon >0$. We know that $q^{-1}(K) \subseteq G$ is a compact subset of $G$.  So there is one $\alpha_0$ so that for each $\alpha \succeq \alpha_0$, $|u_\alpha(y) - 1| < \epsilon$ for all $y \in q^{-1}(K)$.  Note that in particular, for each $x\in K$, $O_x \subseteq q^{-1}(K)$ where $O_x$ denotes the orbit of $x$ in $G$. 

Let $\pi$ denote the regular averaging  $C_0(G) \rightarrow C_0(G)$ into radial functions on $G$ (look at \cite[Definition~2.6]{mu2}). Now we have
\begin{eqnarray*}
| \pi(u_\alpha)(\dot{x}) - 1| &=& |\langle u_\alpha - 1, \pi^*(\dot{x})\rangle |  \leq \int_{O_x} |u_\alpha(y)- 1| d \pi^*(\dot{x})(y) < \epsilon.
\end{eqnarray*}
Note that the last inequality holds as $\pi^*$ is a norm-decreasing map from $M(G) \rightarrow M(G)$.  Therefore, $(v_\alpha)$ satisfies the conditions of \cite[Lemma~4.4]{sk} and therefore, $H$ satisfies $(P_2)$.
\end{proof}

\begin{theorem}\label{t:Leptin-for-ultraspherical}
Let $H$ be an ultraspherical hypergroup on a locally compact group $G$. Then the following conditions are equivalent.
\begin{itemize}
\item[(i)]{$H$ satisfies $(P_2)$.}
\item[(ii)]{$A(H)$ has a bounded approximate identity.}
\item[(iii)]{$A(G)$ has a  $\pi$-radial bounded approximate identity.}
\item[(iv)]{$G$ is amenable.}
\end{itemize}
In this case, $MA(H)=M_{cb}A(H)=B_\lambda(H)$.
\end{theorem}

\begin{proof}
$(i) \Rightarrow (ii)$ is proved in Proposition~\ref{p:Leptin-P-hypergroup}.

To prove this $(ii) \Rightarrow (iii)$ we apply the facts that $A(H)$ is a regular semisimple Banach algebra which can be identified by $A(G)^\sharp$. Suppose that $A(H)$ has a bounded approximate identity norm bounded by $c>0$.
So by \cite[Lemma~2]{ma4}, for each compact set $K \subseteq H$, there is an element $u_{K}$ so that $\norm{u_{K}}_{A(H)} \leq c+1$ so that $u_{K}|_K\equiv 1$.  For each   $K$, define $v_{K}:=\iota(u_{K})$. Note that $v_{K, \epsilon}|_{q^{-1}(K)}$ is constantly one. Let $v$ be an arbitrary element in $A(G) \cap C_c(G)$, let $F= \supp(v) \subset G$. It is immediate that $F \subset K_0$ for $K_0:=q^{-1} (q (F))$. Therefore, $v v_{K_0} = v$. Let $(v_K)_{K \subset H}$ be directed reversely by the inclusion of $K$. Since $C_c(G) \cap A(G)$ is dense in $A(G)$,  $(v_K)_K$, as a bounded net in $A(G)$, it forms a bounded approximate identity for $A(G)$.

$(iii)\Rightarrow (iv)$ is based on the classical Leptin theorem for locally compact groups.

$(iv) \Rightarrow (i)$ is proved in Lemma~\ref{l:G-amenable=>H-P2}.

The equality  $MA(H)=B_\lambda(H)$ for amenable group $G$ was proved in  \cite[Theorem~4.2]{mu2}. Also, since $H$ is an operator Fourier hypergroup, we have that $A(H)$ is a completely contractive Banach algebra with a bounded approximate identity. Therefore, one can use \cite[Lemma~3.3]{att} to prove that $MA(H)=M_{cb}A(H)$. 
\end{proof}

Note that $(iv) \Rightarrow (iii)$ in Theorem~\ref{t:Leptin-for-ultraspherical} was first proved in \cite[Lemma~3.7]{ultra} using basic properties of ultrashperical hypergroups.

\section*{Acknowledgment}
The author is grateful to Jason Crann for a number of fruitful conversations on the topic of this paper. The author was partially supported by a Carleton-Fields Institute Postdoctoral Fellowship.


\vskip1.5em

\end{document}